\documentclass[11pt]{article}
\usepackage{bm}
\usepackage{amsmath,amsthm,amssymb,url}
\usepackage{multirow,bigdelim}
\usepackage{amscd}
\newtheorem{thm}{Theorem}[section]
\newtheorem{theorem}[thm]{Theorem}
\newtheorem{proposition}[thm]{Proposition}
\newtheorem{lemma}[thm]{Lemma}

\newtheorem{claim}[thm]{Claim}
\newtheorem{corollary}[thm]{Corollary}

\newtheorem{conjecture}{Conjecture}
\newtheorem{prop}[thm]{Proposition}

\usepackage{tikz}
\usetikzlibrary{shapes.geometric, calc}

\title{Rigidity of frameworks on expanding spheres}
\author{Anthony Nixon\\
Department of Mathematics and Statistics, Lancaster University\\ Lancaster, LA1 4YF, U.K.\\
 \\ Bernd Schulze\\ 
  Department of Mathematics and Statistics, Lancaster University\\ Lancaster, LA1 4YF, U.K.\\
\\  Shin-ichi Tanigawa\footnote{Supported by JSPS Grant-in-Aid for Scientific Research(A)(25240004)}\\ 
Research Institute for Mathematical Sciences, Kyoto University \\ Sakyo-ku, Kyoto 606-8502, Japan and \\
Centrum Wiskunde \& Informatica (CWI), Postbus 94079, 1090 GB\\
Amsterdam, The Netherlands\\
\\Walter Whiteley\footnote{Supported by a grant from NSERC (Canada)}\\
 Department of Mathematics and Statistics, York University\\ 4700 Keele Street, Toronto, ON M3J 1P3, Canada\\}

\begin{document}

\renewcommand{\thefootnote}{ }

\maketitle

\begin{abstract}
A rigidity theory is developed for bar-joint frameworks in $\mathbb{R}^{d+1}$ whose vertices are
constrained to lie on concentric $d$-spheres with independently variable radii.    
In particular, combinatorial characterisations are established for the rigidity of generic frameworks for $d=1$ with an arbitrary number of independently variable radii, and for $d=2$ with at most two variable radii. This includes a characterisation of the rigidity or flexibility of uniformly expanding spherical frameworks in $\mathbb{R}^{3}$.
Due to the equivalence of the generic rigidity between Euclidean space and spherical space, these results interpolate between rigidity in 1D and 2D and to some extent between rigidity in 2D and 3D. Symmetry-adapted counts for the detection of symmetry-induced continuous flexibility in frameworks on spheres with variable radii are also provided.
\end{abstract}

\section{Introduction}

Let $G=(V,E)$ be a finite, simple graph. A \emph{framework} $(G,p)$ in $\mathbb{R}^d$ is a geometric realisation of $G$ via a map $p:V\rightarrow \mathbb{R}^d$. $(G,p)$ is \emph{rigid} if every edge-length preserving motion of $(G,p)$ arises as an isometry of $\mathbb{R}^d$. In general it is an NP-hard problem to determine the rigidity of a given framework \cite{Abb}. However for generic frameworks 
 the situation improves. When $d\leq 2$ there are simple combinatorial characterisations of generic rigidity (see \cite{Lam}, for example) that lead to efficient polynomial time algorithms. When $d\geq 3$ there are simple counterexamples to the natural analogue of these results \cite{Wlong} and it is a fundamental open problem in discrete geometry to determine if any combinatorial characterisations are possible.

Rigidity is also well studied for a variety of geometric situations. In particular, consider frameworks in $\mathbb{R}^{d+1}$ whose vertices are constrained to lie on a fixed $d$-dimensional sphere $\mathbb{S}^d$. It was explicit in a paper on coning by Whiteley \cite{Wcones} and extended by Saliola and Whiteley \cite{SaliolaWhiteley} (see also Izmestiev \cite{i}) that there is a one-to-one correspondence between the infinitesimal rigidity of $(G,p)$ in $\mathbb{R}^d$ and infinitesimal rigidity  of  $(G,p^*)$ on the sphere  $\mathbb{S}^d$ where $p^*$ projects from the center of the sphere to $p$. These connections among rigidity of frameworks in various metrics were initiated in work by Pogorolev \cite{pog}. 

In this paper we consider a variant of spherical rigidity, that is, the rigidity/flexibility of expanding spherical frameworks where the radius of the spherical framework may change continuously. 
Our initial motivation is to understand the spherical mechanisms of mathematical toys inspired by popular structures like the Hoberman sphere or Buckminster Fuller's `jitterbug' (such as Juno's spinners \cite{juno,kov}, for example). Rigidity and flexibility analyses of expanding spherical structures also have significant practical applications. For example, an important problem in biochemistry is to understand the swelling motions of virus capsids \cite{expandoh}.

In fact our results will be established in a more general setting. 
One can quickly see that generic rigidity on $\mathbb{S}^d$ is equivalent to generic rigidity on concentric $d$-spheres with fixed radii.
Thus we consider a further general model of spherical rigidity where the framework points are required to remain on concentric spheres throughout any motion, but  the spheres may vary their respective radii independently of each other in any continuous motion. To this end our first main result gives a combinatorial description of generic rigidity for such frameworks when $d=1$ (Theorem \ref{thm:1d}). Observe that the case when each vertex lies on its own variable sphere  corresponds to standard (Euclidean) rigidity in $\mathbb{R}^{d+1}$. Therefore our main theorem (Theorem \ref{thm:1d}) interpolates between rigidity on the line and the plane. 

Similarly, extending Theorem \ref{thm:1d} to the case $d=2$ would solve the famous\break 3-dimensional rigidity problem. 
Our second main result (Theorem \ref{thm:2d}) is a partial extension of Theorem \ref{thm:1d} which gives a combinatorial characterisation of rigidity for $d=2$ when there are at most two independent rates of variability among the radii. This includes the special case  of uniformly expanding spherical frameworks. Further extensions for $d=2$, however, seem difficult, and in Section~\ref{sec:higherdim} we illustrate counterexamples to the sufficiency of our counting condition when there are at least three independent rates.

In Section~\ref{sec:sym} we will also analyse symmetric spherical frameworks since  many of the man-made and biological structures discussed above possess non-trivial symmetries (typically rotational polyhedral symmetries).

We conclude the introduction with a short outline of what follows. In Section \ref{sec:variable} we formally describe our rigidity problem in terms of coloured graphs and derive the appropriate necessary (Maxwell-type) count conditions for rigidity. 
Section~\ref{sec:prelim} gives some preliminary results showing that our count functions are monotone and submodular, and hence induce matroids on the edge sets of the graphs. In Section~\ref{sec:prelim}  we also develop a geometric understanding of the well known 0- and 1-extension operations (also known as Henneberg operations) applied to our setting. Everything in Section \ref{sec:prelim} is independent of $d$. Section \ref{sec:comb} is devoted to the proof of our main results. Note that for $d=1$ and any fixed $k$ it is plausible that a Henneberg-Laman type recursive construction can be derived, and for $d=2$ with $k=1$ this is certainly true. However for $d=1$ with arbitrary $k$, and for $d=2$ with $k=2$ such methods seem to be difficult. Instead, we give a proof by induction on both the number of vertices and colours, making extensive use of the properties of the count functions and the geometric results from Section~\ref{sec:prelim}.
 In Section \ref{sec:sym} we then study the impact of symmetry on the rigidity properties of frameworks on variable spheres. In particular, we  provide symmetry-adapted combinatorial counts to detect symmetry-induced flexibility in such frameworks.
 Finally, in Section~\ref{sec:concl} we make some concluding remarks.

\section{Spheres with variable radii }
\label{sec:variable}
In this section we formally describe our rigidity problem on spherical frameworks with variable radii.
The combinatorics of such frameworks will be captured by vertex-coloured graphs,
where each colour will correspond to an independent rate of variability in the radius of the sphere.

Let $G=(V,E)$ be a simple graph with $V=\{1,\dots, n\}$.
By a {\em colouring} we mean a function $\chi:X\rightarrow {\cal C}$ from a subset $X$ of $V$ to a set ${\cal C}$ of colours, and a pair $(G,\chi)$ is said to be a coloured graph. 
A vertex not in the support of $\chi$ is said to be an {\em uncoloured vertex}.
We may always assume that $\chi$ is surjective.
Also it is convenient to introduce a special sign $\bullet$ to denote the uncoloured situation,
and regard $\chi$ as a function from $V$ to ${\cal C}\cup \{\bullet\}$
A colouring is called a $k$-colouring if the size of the image of $\chi$ is equal to $k$.
For a colour $c\in {\cal C}$, $\chi^{-1}(c)$ denotes the set of vertices having colour $c$. 

For $t\in \mathbb{R}_{>0}$ let $t\mathbb{S}^d$ be the sphere with radius $t$ centred at the origin. 
Suppose there is a $k$-coloured graph $(G=(V,E), \chi)$
and $r:{\cal C}\rightarrow \mathbb{R}_{>0}$.
We denote by $r\mathbb{S}^d$ the family of concentric spheres $\{r(c)\mathbb{S}^d \mid c\in {\cal C}\}$,
and by {\em a framework $(G,\chi, p)$ on $r\mathbb{S}^d$} we mean a tuple of $G, \chi$ and $p:V\rightarrow \mathbb{R}^{d+1}$ such that all joints coloured $c$ lie on $r(c)\mathbb{S}^d$
and all uncoloured joints lie on $\mathbb{S}^d$. 
We say that the framework $(G,\chi, p)$ is \emph{generic} on $r\mathbb{S}^d$ if the coordinates of $p$ and $r$ do not satisfy any nonzero polynomial except for those belonging to the ideal generated by  the defining polynomials of concentric spheres $r\mathbb{S}^d$. That is, the coordinates of $p$ are as algebraically independent as possible given that $(G,\chi, p)$ lies on $r\mathbb{S}^d$.

For such a framework in $\mathbb{R}^{d+1}$ we are interested in the rigidity with respect to the following motions: each joint is allowed to move under the condition that each edge-length is fixed,  
all joints in $\chi^{-1}(c)$ lie on a sphere (whose center is the origin) for each colour $c\in {\cal C}$,
and each uncoloured joint lies on the unit sphere. This can easily be formalised into a definition of continuous rigidity and using the standard Asimow-Roth technique \cite{A&R} can be shown to be, generically, equivalent to infinitesimal rigidity (defined below). Since all our proofs are in terms of infinitesimal rigidity we concentrate on the infinitesimal theory.

Namely, given a framework $(G,\chi,p)$ on $r\mathbb{S}^d$ with a $k$-colouring $\chi$,
we are interested in the dimension of the space of infinitesimal motions, 
where an \emph{infinitesimal motion} means a pair $(\dot{p},\dot{r})$ where $\dot{p}:V\rightarrow \mathbb{R}^{d+1}$ 
and $\dot{r}:{\cal C}\rightarrow \mathbb{R}$ satisfy the following system of equations:
\begin{align}
(p(u)-p(v))\cdot (\dot{p}(u)-\dot{p}(v))=0 & \qquad (uv\in E), \label{eq:const1}\\
p(v)\cdot \dot{p}(v)=\dot{r}(c)r(c) & \qquad (c\in {\cal C}\cup \{\bullet\}, v\in \chi^{-1}(c)), \label{eq:const2}
\end{align}
with $\dot{r}(\bullet)=0$.
By resetting $\dot{r}(c)$ to be $r(c)\dot{r}(c)$ we shall consider 
\begin{equation}
\label{eq:const3}
p(v)\cdot \dot{p}(v)=\dot{r}(c)
\end{equation}
instead of (\ref{eq:const2}).
The \emph{rigidity matrix} $R(G,\chi,p)$ is a matrix of size $(|E|+n) \times ((d+1)n+k)$ representing the system 
(\ref{eq:const1}), (\ref{eq:const3}).
More precisely, $R(G,\chi,p)$ consists of four blocks:
\begin{equation}\label{eq:rigmat}
R(G,\chi,p)=\begin{pmatrix}
R(G,p) & 0 \\
S(G,p) & V(G,\chi, p)
\end{pmatrix},
\end{equation}
where each block is defined as follows.
$R(G,p)$ is the usual rigidity matrix for a framework in $\mathbb{R}^{d+1}$,
where the rows are indexed by $E$ and sets of $d+1$ consecutive columns are indexed by $V$, and the entries in the row of edge $e = ij$ and in the $d+1$ columns of $i$ and $j$ contain the $d+1$ coordinates of $p(i)-p(j)$ and $p(j)-p(i)$, respectively, and the remaining entries are zeros.
$S(G,p)$  has size $|V|\times (d+1)|V|$ and is written as 
\begin{equation*}
S(G,p)=\begin{pmatrix}
p(1) &&&0 \\
 & p(2) && \\
 && \ddots & \\
0 &&& p(n) 
\end{pmatrix}
\end{equation*}
where the rows are indexed by $V$ and sets of $d+1$ consecutive columns are also indexed by $V$.
$V(G,\chi,p)$ has size $|V|\times |{\cal C}|$, where the rows are indexed by $V$ and the columns are indexed by ${\cal C}$, and the entries in the row of coloured vertex $v$ and in  the column of its colour $\chi(v)$ contain the coordinate $1$, and the remaining entries are zeros.

Suppose that $p(V)$ affinely spans $\mathbb{R}^{d+1}$.
Then, as we will see below, the rank of $R(G,\chi,p)$ is  bounded from above by $(d+1)n-{d+1\choose 2}+\min\{k,n-(d+1)\}$. We say that $(G,\chi,p)$ is \emph{infinitesimally rigid} if ${\rm rank}\ R(G,\chi,p)=(d+1)n-{d+1\choose 2}+\min\{k,n-(d+1)\}$. A framework $(G,\chi,p)$ is called \emph{independent}  if the rows of $R(G,\chi,p)$ are linearly independent, and $(G,\chi,p)$ is called \emph{isostatic} if it is independent and infinitesimally rigid.
 
Notice that, if all vertices are uncoloured, the concept coincides with the infinitesimal rigidity of spherical frameworks with fixed radius, where we are concerned with the rank of 
\begin{equation*}
R_{\mathbb{S}^d}(G,p):=\begin{pmatrix}
R(G,p)  \\
S(G,p) 
\end{pmatrix}.
\end{equation*}
It is known \cite{SaliolaWhiteley,i} that ${\rm dim}\ {\rm ker}\ R_{\mathbb{S}^d}(G,p)={\rm dim}\ {\rm ker}\ R(G,p')$ 
for a $d$-dimensional framework $(G,p')$ obtained from $(G,p)$ by projecting to a hyperplane.
This means that if all vertices are uncoloured, our infinitesimal rigidity theory  has a one-to-one correspondence with the infinitesimal rigidity in $\mathbb{R}^d$.
On the other hand, if each vertex has a distinct colour, then the concept coincides with the infinitesimal rigidity in $\mathbb{R}^{d+1}$. So this concept ``interpolates" the rigidity between $\mathbb{R}^d$ and $\mathbb{R}^{d+1}$.
 
\begin{proposition}
\label{prop:trivial}
Let $(G,\chi,p)$ be a framework on $r\mathbb{S}^d$.
Suppose that $p(V)$ affinely spans $\mathbb{R}^{d+1}$. Then 
\begin{equation}
{\rm rank}\ R(G,\chi,p)\leq {\rm rank}\ R_{\mathbb{S}^{d}}(G,\overline{p})+\min\{k,n-(d+1)\},
\end{equation}
where $(G,\overline{p})$ denotes the  framework on $\mathbb{S}^d$ obtained from $(G,p)$ by central projection.
\end{proposition}
\begin{proof}
By associating  $\dot{p}\in {\rm ker} \begin{pmatrix} R(G,p) \\ S(G,p)\end{pmatrix}$ 
with $\dot{q}\in {\rm ker} R_{\mathbb{S}^d}(G,\bar{p})$  
such that $\dot{q}(i)=\dot{p}(i)/\|p(i)\|$ for $i\in V$, we have a linear bijection between the two kernels,
i.e.,
\[{\rm dim }\ {\rm ker} R_{\mathbb{S}^d}(G,\overline{p})={\rm dim }\ {\rm ker} \begin{pmatrix}
R(G,p) \\ S(G,p)\end{pmatrix}.\]
For each $\dot{p}\in {\rm ker} \begin{pmatrix}
R(G,p) \\ S(G,p)\end{pmatrix}$, a pair $(\dot{p},\dot{r}=0)$ is an infinitesimal motion of $(G,\chi, p)$.
Hence $${\rm rank}\ R(G,\chi,p)\leq {\rm rank}\ R_{\mathbb{S}^{d}}(G,\overline{p})+k.$$

Suppose $k\geq n-(d+1)$.
For each colour $c$, take a representative vertex in $\chi^{-1}(c)$, denoted by $v_c$.
For any $x\in \mathbb{R}^{d+1}$, the pair $(\dot{p}, t)$ defined 
by $\dot{p}(v)=x$ for $v\in V$ and $\dot{r}(c)=p(v_c)\cdot x$ for each colour $c$ is an infinitesimal motion if 
\begin{align}
\label{eq:prop1-1}
(p(v_c)-p(u))\cdot x&=0 \qquad \forall c\in {\cal C}, \forall u\in \chi^{-1}(c)\setminus \{v_c\}, \\
p(v)\cdot x&=0 \qquad \forall \text{ uncoloured vertex } v. \label{eq:prop1-2}
\end{align}
This gives a system of $(n-k)$ linear equations on $x\in \mathbb{R}^{d+1}$.
Hence there are at least $d+1-(n-k)$ linearly independent choices for $x$.

Observe that, if $p(V)$ spans $\mathbb{R}^{d+1}$, the space of infinitesimal motions derived from this construction 
has a zero intersection with the space of infinitesimal motions derived from the kernel of $\begin{pmatrix} R(G,p) \\ S(G,p)\end{pmatrix}$ since the former is a subspace of the space of infinitesimal translations while the latter consists of (scaled) infinitesimal motions on the sphere with fixed radius.
In other words, if $d+1\geq n-k$, then the dimension of the infinitesimal motion is 
at least 
\[
(d+1)n-{\rm rank}\ R_{\mathbb{S}^{d}}(G,p)+d+1-(n-k).
\]
Hence the rank is bounded above by 
\[
(d+1)n+k-[(d+1)n-{\rm rank}\ R_{\mathbb{S}^{d}}(G,p)+d+1-(n-k)]=
{\rm rank}\ R_{\mathbb{S}^{d}}(G,p)+n-(d+1).
\]
\end{proof}

For a skew-symmetric matrix $S$ of size $(d+1)\times (d+1)$, 
define $\dot{p}:V\rightarrow \mathbb{R}^{d+1}$ by $\dot{p}(v)=Sp(v)$ for $v\in V$.
Then the pair $(\dot{p}, \dot{r}=0)$ is an infinitesimal motion of $(G,\chi,p)$.
Also if $k\geq n-(d+1)$ there is an infinitesimal motion for each $x\in \mathbb{R}^{d+1}$ 
satisfying (\ref{eq:prop1-1})(\ref{eq:prop1-2}) as given in the proof of Proposition~\ref{prop:trivial}.
A linear combination of those infinitesimal motions is said to be {\em trivial}.

\section{Preliminary observations}
\label{sec:prelim}

In this section we record some preliminary results that we shall need of the proof of our main theorem in Section~\ref{sec:comb}.

\subsection{Submodular functions}

For $F\subseteq E$, 
let $V(F)$ be the set of vertices spanned by $F$ and 
$k(F)$ be the number of colours in $V(F)$.
Proposition~\ref{prop:trivial} implies the following.
\begin{lemma}
\label{lem:necessity}
Let $n\geq d+1$ and let $(G,\chi,p)$ be a generic framework whose underlying graph $(G,\chi)$ is  $k$-coloured.
If $(G,\chi,p)$ is isostatic on $r\mathbb{S}^d$, then 
$|E|=dn-{d+1\choose 2}+\min\{k,n-(d+1)\}$ and $|F|\leq r_d(F)+\min\{k(F),|V(F)|-(d+1)\}$ for any $F\subseteq E$ 
with $|V(F)|\geq d+1$, 
where $r_d$ denotes the rank function of the $d$-dimensional generic rigidity matroid.
\end{lemma}
Our goal is to show the converse direction of Lemma~\ref{lem:necessity}. 
For this purpose it is convenient to know that the set of graphs satisfying the combinatorial count condition in Lemma \ref{lem:necessity} forms the set of bases of a matroid, which will follow from the following matroid construction by submodular functions.

For a finite set $E$, a function $f:2^E\rightarrow \mathbb{R}$ is called \emph{submodular}
if $f(X)+f(Y)\geq f(X\cup Y)+f(X\cap Y)$ for any $X, Y\subseteq E$ 
while $f$ is called \emph{monotone} if $f(X)\leq f(Y)$ for $X\subseteq Y\subseteq E$.
A monotone submodular function $f$ {\em induces} a matroid on $E$, 
where $I\subseteq E$ is independent  if and only if $|F|\leq f(F)$ for any nonempty $F\subseteq I$(\cite{E70}).
Suppose that $E$ is the edge set of a graph $G$. 
Then $G$ is said to be \emph{$f$-sparse} if $E$ is independent in the matroid induced by $f$,
and $G$ is said to be \emph{$f$-tight} if $G$ is $f$-sparse with $|E|=f(E)$.

In our problem, for a graph $G=(V,E)$ with $\chi$, we consider the following functions $h_d:2^E\rightarrow \mathbb{Z}$ 
and $g_d:2^E\rightarrow \mathbb{Z}$ defined by
\begin{equation*}
\begin{array}{ll}
h_d(F)=\min\{k(F),|V(F)|-(d+1)\}  &   (F\subseteq E) \\
g_d(F)=r_d(F)+h_d(F) & (F\subseteq E).
\end{array}
\end{equation*}
We remark that $k$ is submodular as 
\begin{equation}
\label{eq:k}
\begin{split}
k(F)+k(F')&=k(F\cup F')+\text{(\# colours common in $V(F)$ and $V(F')$)} \\
&\geq k(F\cup F')+k(F\cap F').
\end{split}
\end{equation}

\begin{lemma}
\label{lemma:sub}
The function $g_d$ is monotone and submodular.
\end{lemma}
\begin{proof}
As the sum of two monotone submodular functions is monotone submodular, it suffices to show 
the property for $h_d$.
The monotonicity is obvious.
To see the submodularity we shall check $h_d(X+e)-h_d(X)\geq h_d(Y+e)-h_d(Y)$ for any $X\subseteq Y\subseteq E$ and $e\in E\setminus Y$,
which is known to be equivalent to the submodularity.

Note that $h_d(X+e)-h_d(X)\leq 2$ and $h_d(Y+e)-h_d(Y)\leq 2$.
Hence it suffices to consider the case when $h_d(X+e)-h_d(X)\leq 1$.

Suppose first that $k(X)> |V(X)|-(d+1)$.
Then by $h_d(X+e)-h_d(X)\leq 1$ we have $|V(X+e)|-(d+1)\leq |V(X)|-d\leq k(X)\leq k(X+e)$.
Hence $h_d(X+e)-h_d(X)=|V(X+e)|-|V(X)|$.
If $k(Y)\geq |V(Y)|-(d+1)$, then $h_d(Y)=|V(F)|-(d+1)$ and the submodularity of $|V(\cdot)|$ implies that
$|V(X+e)|-|V(X)|\geq |V(Y+e)|-|V(Y)|\geq h_d(Y+e)-h_d(Y)$.
On the other hand if $k(Y)\leq |V(Y)|-(d+1)$ then 
$k(Y+e)\leq |V(Y+e)|-(d+1)$.
Hence $|V(X+e)|-|V(X)|\geq k(Y+e)-k(Y)=h_d(Y+e)-h_d(Y)$.

Suppose next that $k(X)\leq |V(X)|-(d+1)$.
Note that $k(Y)\leq k(X)+|V(Y)\setminus V(X)|$.
Hence we have $k(Y)\leq |V(Y)|-(d+1)$.
Note also that if $k(X)\leq |V(X)|-(d+1)$, then $k(X+e)\leq |V(X+e)|-(d+1)$.
Therefore, 
$h_d(X+e)-h_d(X)=k(X+e)-k(X)\geq k(Y+e)-k(Y)\geq h_d(Y+e)-h_d(Y)$,
where the second inequality follows from the submodularity of $k$.
\end{proof}


\subsection{Geometric operations}

In a $k$-coloured graph $(G,\chi)$, a vertex of colour $c$ is said to be \emph{colour-isolated} if there is no other vertex of colour $c$.
We consider the following two operations that reduce a coloured graph to a smaller one.

A ($d$-dimensional) {\em 0-reduction} removes a vertex $v$ with degree equal to $g_d(E)-g_d(E(G-v))$, where $E(G-v)$ denotes the edge set of $G-v$. 
More specifically, if $k\geq |V(G)|-(d+1)$ then 
a 0-reduction removes a vertex of degree $d+1$,
and if $k< |V(G)|-(d+1)$ then a 0-reduction removes 
a non-colour-isolated vertex or an uncoloured vertex of degree $d$, or  a colour-isolated vertex of degree $d+1$.

A ($d$-dimensional) {\em 1-reduction} removes a non-colour-isolated vertex of degree $d+1$, an uncoloured vertex of degree $d+1$, or a colour-isolated vertex of degree $d+2$, and then inserts a new edge between two distinct vertices adjacent to the removed vertex in the original graph.
The reverse operation of $0$- or $1$-reduction is called {\em $0$-  or $1$-extension}. 

\begin{lemma}
\label{lem:0extension}
Let $(G,\chi,p)$ be a generic isostatic framework on $r\mathbb{S}^d$ and let $(G',\chi')$ be formed from $(G,\chi)$ by a 0-extension. Then any generic realisation of $(G',\chi')$ is isostatic on $r\mathbb{S}^d$.
\end{lemma}
\begin{proof} 
Let $v$ be the new vertex and let $(G',\chi',p')$ be a generic realization obtained by extending $(G,\chi, p)$.
If $v$ is colour-isolated, $R(G',\chi',p')$ has $d+1$ new rows for the edges incident to $v$ and the new row for $v$. Since the number of columns is increased by $d+2$, one can easily check the row independence of $R(G',\chi',p')$ from the definition.

Hence assume that $v$ is not colour-isolated.
Let $k'$ be the number of colours in $G'$.
If $k'<|V(G')|-(d+1)$, then the number of new edges is $d$, and  it is straightforward to check the row independence of $R(G',\chi',p')$.
Hence assume $k'\geq |V(G')|-(d+1)$.
Take any infinitesimal motion $(\dot{p},\dot{r})$ of $(G',\chi',p')$.
Since $(G,\chi,p)$ is infinitesimally rigid, the restriction of $(\dot{p},\dot{r})$ to $(G,\chi,p)$ turns out to be trivial. Moreover, since $d+1$ new edges are added, $(\dot{p},\dot{r})$ is trivial.
Now observe that the dimension of the trivial motions decreases by one
 since the dimension of the solution space of the linear system (\ref{eq:prop1-1})(\ref{eq:prop1-2}) decreases by one when $p'$ is generic.   
This means that ${\rm dim}\ {\rm ker}\ R(G',\chi',p')={\rm dim}\ {\rm ker}\ R(G,\chi, p)-1$.
Thus ${\rm rank}\ R(G',\chi',p')-{\rm rank}\ R(G,\chi,p)=d+2$.
In other words $(G',\chi',p')$ is infinitesimally rigid.
\end{proof}

For a graph $G'$ and a vertex $v$ in $G'$ let $N_{G'}(v)$ be the set of neighbors of $v$ in $G'$.
\begin{lemma}
\label{lem:1extension1}
Let $(G,\chi,p)$ be a generic isostatic framework on $r\mathbb{S}^d$ and let $(G',\chi')$ be formed from $(G,\chi)$ by a 1-extension. Suppose that the new vertex $v$ is colour-isolated or 
all vertices in $N_{G'}(v)\cup \{v\}$ are uncoloured. Then any generic realisation of $(G',\chi')$ is isostatic on $r\mathbb{S}^d$.
\end{lemma}
\begin{proof}
Let $v$ be the new vertex.
If $v$ is colour-isolated, we may freely choose $r(\chi(v))$.
Hence we can prove the statement by using the standard collinear triangle technique (see~\cite[Theorem 2.2.2]{Wlong} for example).

If $N_{G'}(v)\cup \{v\}$ are uncoloured, the situation is the same as the conventional rigidity on the fixed sphere,
and can be solved by extending the collinear triangle technique.
(Alternatively one can also apply the limit argument given in \cite{nop2}.)
\end{proof}

For non-colour-isolated vertices we use a technique reminiscent of \cite[Section 6]{nop2}. This can be used to prove that ($d$-dimensional) 1-extension preserves rigidity in the case $k=1$. However our proof technique cannot be applied for general coloured graphs, and instead we  have the following statement.

\begin{lemma}
\label{lem:1extension2}
Let $(G,\chi,p)$ be a generic framework on $r\mathbb{S}^d$
and $v$ be a non-colour-isolated vertex of degree more than $d$.
Suppose that both $(G,p)$ and $(G-v,p)$ are one-degree of freedom frameworks.
Then there is a nontrivial infinitesimal motion  $(\dot{p}, \dot{r})$ of $(G,\chi,p)$ such that 
$\dot{r}(\chi(v))=0$ and $\dot{r}(\chi(u))=0$ for any $u\in N_G(v)$. 
\end{lemma}
\begin{proof}
The vector $(\dot{p}, \dot{r})$ is an infinitesimal motion iff 
\begin{align*}
\langle p(v), \dot{p}(v)\rangle&=\dot{r}(\chi(v)) \qquad \forall v\in V \mbox{ and} \\
\langle p(u), \dot{p}(v)\rangle+\langle p(v), \dot{p}(u)\rangle&=\dot{r}(\chi(u))+\dot{r}(\chi(v)) \qquad  \forall uv\in E(G),
\end{align*}
where $\dot{r}(\chi(v))=0$ if $v$ is uncoloured. Without loss of generality we may assume that the vertices $\{1,\dots, {d+1}\}$ are neighbors of $v$ in $G$. 

Let $p(V-v)=\{p(v'):v'\in V-v\}$ and
let $\mathbb{Q}(p(V-v))$ be the field generated by the rationals and the entries in $p(V-v)$. Note that the set of  coordinates of $p_v$ is generic  over $\mathbb{Q}(p(V-v))$ in the sphere, i.e., 
any vanishing polynomial with coefficients in $\mathbb{Q}(p(V-v))$ is in the ideal generated by 
$p(v)^{\top} p(v) -r(\chi(v))^2$.

Since $v$ has degree at least $d+1$ and is not colour-isolated, a trivial infinitesimal motion of $(G-v,p)$ cannot be extended to a nontrivial infinitesimal motion of $(G,p)$. 
In other words the restriction of a nontrivial infinitesimal motion  $(\dot{p},\dot{r})$ of $(G,p)$ onto $V-v$ is nontrivial for $(G-v, p)$. 
Since both $(G,p)$ and $(G-v,p)$ have one-degree of freedom and
 a nontrivial infinitesimal motion of $(G-v,p)$ is obtained by solving a system of linear equations with coefficients in $\mathbb{Q}(p(V-v))$,  we may further assume that each coordinate of $\dot{p}(V-v)$ and $\dot{r}$ is algebraic over $\mathbb{Q}(p(V-v))$
(where we used the fact that $v$ is not colour-isolated to see that $\dot{r}(\chi(v))$ is also algebraic over 
$\mathbb{Q}(p(V-v))$).

Let $A$ be the $(d+1)\times (d+1)$-matrix whose $i$-th row is $p(i)^{\top}$.
$A$ is nonsingular. Let $\dot{A}$ be $(d+1)\times (d+1)$-matrix whose $i$-th row is $\dot{p}(i)^{\top}$.
Let ${\bm t}$ be the $(d+1)$-dimensional vector whose $i$-th coordinate is $\dot{r}(\chi(i))$,
and let $\mathbf{1}$ be the all-one vector. Now, from the condition for $(\dot{p},\dot{r})$ to be an infinitesimal motion, we have
\begin{equation*}
A\dot{p}(v)+\dot{A}p(v)=\dot{r}(\chi(v))\mathbf{1}+{\bm t}
\end{equation*}
\begin{equation*}
\langle p(v), \dot{p}(v)\rangle=\dot{r}(\chi(v)).
\end{equation*}
Hence 
\begin{equation*}
\dot{p}(v)=A^{-1}(\dot{r}(\chi(v))\mathbf{1}+{\bm t}-\dot{A}p(v))
\end{equation*}
and 
\begin{equation}
\label{eqn:1-ex}
p(v)^{\top}A^{-1}\dot{A}p(v)-\dot{r}(\chi(v))p(v)^{\top}A^{-1}\mathbf{1}-p(v)^{\top}A^{-1}{\bm t}+\dot{r}(\chi(v))=0.
\end{equation}
Let $s_i=\frac{\|p(v)\|}{\|p(i)\|}$. Due to the choice of $p(v)$, the left hand side polynomial can be divided by $p^{\top}(v)p(v)-r(\chi(v))^2$. 
Therefore  (\ref{eqn:1-ex}) holds by substituting $p(v)$ with any $x\in \mathbb{R}^{d+1}$ satisfying $x^{\top}x=r(\chi(v))^2$.

In particular, since $(-s_ip(i))^{\top} (-s_ip(i))=r(\chi(v))^2$, by setting $p(v)=-s_ip(i)$ we have 
\begin{equation*}
s_i^2p(i)^{\top}A^{-1}\dot{A}p(i)+s_i\dot{r}(\chi(v))p(i)^{\top}A^{-1}\mathbf{1}+s_ip(i)^{\top}A^{-1}{\bm t}+\dot{r}(\chi(v))=0.
\end{equation*}
Note that for each $i=1,\dots, d+1$ we have 
\begin{align*}
p(i)^{\top}A^{-1}\dot{A}p(i)&=\langle \dot{p}(i), p(i)\rangle=\dot{r}(\chi(i)), \\ 
p(i)^{\top}A^{-1}\mathbf{1}&=1 \mbox{ and } \\
p(i)^{\top}A^{-1}{\bm t}&=\dot{r}(\chi(i)).
\end{align*}
Hence we get $(s_i+1)(s_i\dot{r}(\chi(i))+\dot{r}(\chi(v)))=0.$ Thus 
\begin{equation}
\label{eqn:16}
s_i\dot{r}(\chi(i))+\dot{r}(\chi(v))=0.
\end{equation}

On the other hand by setting $s_{ij}=\frac{\|p(v)\|}{\|p(i)+p(j)\|}$ and $p(v)=-s_{ij}(p(i)+p(j))$ for any $i,j\in \{1,\dots, d+1\}$
we have 
\begin{equation}
2s_{ij}^2(\dot{r}(\chi(i))+\dot{r}(\chi(j)))+s_{ij}(2\dot{r}(\chi(v))+\dot{r}(\chi(i))+\dot{r}(\chi(j)))+\dot{r}(\chi(v))=0
\end{equation}
since
\begin{align*}
(p_i+p_j)A^{-1}\dot{A}(p(i)+p(j))&=\langle p(i), \dot{p}(i)\rangle+\langle p(i),\dot{p}(j)\rangle+\langle p(j),\dot{p}(j)\rangle+\langle p(j),\dot{p}(j)\rangle \\ &=2(\dot{r}(\chi(i))+\dot{r}(\chi(j))) \mbox{ and }\\ (p(i)+p(j))^{\top}A^{-1}(\dot{r}(\chi(v))\mathbf{1}+{\bm t})&=(2\dot{r}(\chi(v))+\dot{r}(\chi(i))+\dot{r}(\chi(j))). 
\end{align*}
Thus $(2s_{ij}+1)(s_{ij}(\dot{r}(\chi(i))+\dot{r}(\chi(j)))+\dot{r}(\chi(v)))=0$ 
and 
\begin{equation}
s_{ij}(\dot{r}(\chi(i))+\dot{r}(\chi(j)))+\dot{r}(\chi(v))=0.
\label{eqn:17}
\end{equation}

Combining Equations (\ref{eqn:16}) and (\ref{eqn:17})
we get 
\begin{equation*}
(\|p(i)\|+\|p(j)\|-\|p(i)+p(j)\|)\dot{r}(\chi(v))=0.
\end{equation*}
Since $p(i), p(j)$  and the origin are not collinear, we get $\dot{r}(\chi(v))=0$ 
and hence $\dot{r}(\chi(u))=0$ for any neighbor $u$ of $v$ by (\ref{eqn:16}).
\end{proof}

\section{Combinatorial characterisations}
\label{sec:comb}

In this section we prove our main theorems, combinatorial characterisations of generic infinitesimal rigidity for $d=1$ with arbitrary $k$ and for $d=2$ with $k\leq 2$.

\subsection{1-dimensional case}

If $d=1$ then $g_1(F)> 0$ for any nonempty $F$.
Hence by Lemma~\ref{lemma:sub} the set of all edge sets satisfying the condition in Lemma~\ref{lem:necessity}
form the bases of a matroid on the edge set of the complete graph on $V$, 
which is denoted by ${\cal M}_1(V,\chi)$. 
The coloured graph whose edge set is a base (resp., independent set) is 
called $g_1$-tight (resp., $g_1$-sparse).

Also recall that  $r_1(F)=n-\omega(F)$ for any $F\subseteq E$, where 
$\omega(F)$ denotes the number of connected components in $(V,F)$.

Our goal in this subsection is to show Theorem~\ref{thm:1d} below, 
which shows that the generic rigidity on $r\mathbb{S}^1$ is characterized by $g_1$-sparsity.
We need the following three lemmas for the proof.

\begin{lemma}
\label{lem:1-1}
A $g_1$-tight $k$-coloured graph $(G,\chi)$ contains a vertex of degree at most two or a colour-isolated vertex of degree three.
\end{lemma}
\begin{proof}
Let $n_i$ be the number of vertices of degree $i$ and let $k_1$ be the number of colour-isolated vertices.
Suppose every vertex has degree at least three. We shall show that there is a colour-isolated vertex of degree three. 
The statement easily follows if $k\geq |V|-2$ since in this case at least $|V|-4$ vertices are colour-isolated.
So assume $k< |V|-2$.
Then $2|E|=2n-2+2k\geq 3n_3+4(n-n_3)$.
Hence $n_3-2+2k\geq 2n$.

Suppose all vertices of degree three are not colour-isolated.
Then $k_1\leq n-n_3$.
Since $k-k_1\leq \frac{n-k_1}{2}$, we get 
\begin{equation*}
k=(k-k_1)+k_1\leq \frac{n-k_1}{2}+k_1
=\frac{n+k_1}{2}
\leq \frac{2n-n_3}{2}
\leq n_3-(n_3-k+1)=k-1,
\end{equation*}
a contradiction.
\end{proof}

For a colour $c\in {\cal C}$ 
we say that $F\subseteq E$ is {\em $(c,d)$-tight} (or, simply $c$-tight if $d$ is clear) if $|F|=r_d(F)+k(F)$ with $c\in \chi(V(F))$. 
\begin{lemma}
\label{lem:1-2}
Suppose $(G,\chi)$ is $g_1$-sparse.
Then for each colour $c$ either there is no $(c,1)$-tight set or 
$G$ contains  a unique inclusionwise minimal $(c,1)$-tight set.
\end{lemma}
\begin{proof}
Suppose there are two distinct inclusionwise minimal $c$-tight sets $F_1$ and $F_2$.

If $F_1\cap F_2=\emptyset$, then (\ref{eq:k}) implies
$k(F_1)+k(F_2)\geq k(F_1\cup F_2)+1$ since $c$ is a common colour.
Hence by the submodularity of $r_1$ we get
\begin{align*}
|F_1\cup F_2|&=|F_1|+|F_2|=r_1(F_1)+k(F_1)+r_1(F_2)+k(F_2)\\
&\geq r_1(F_1\cup F_2)+k(F_1)+k(F_2)\\
&\geq r_1(F_1\cup F_2)+k(F_1\cup F_2)+1\\
&\geq g_1(F_1\cup F_2)+1,
\end{align*}
contradicting the $g_1$-sparsity condition.

If $F_1\cap F_2\neq \emptyset$,
by the submodularity of $g_1$ 
we have 
\begin{align*}|F_1\cup F_2|+|F_1\cap F_2|&=|F_1|+|F_2| \geq
g_1(F_1)+g_1(F_2)\\ &\geq g_1(F_1\cup F_2)+g_1(F_1\cap F_2)\\
&\geq |F_1\cup F_2|+|F_1\cap F_2|.
\end{align*}
Hence the equality holds everywhere.
In particular, since $r_d$ and $h_d$ are submodular and $g_1=r_1+h_1$, we have 
\begin{align}
\label{eq:lem1-1-1}
|F_1\cap F_2|&=g_1(F_1\cap F_2) \mbox{ and } \\ \label{eq:lem1-1-2}
k(F_1)+k(F_2)&=\min\{k(F_1\cup F_2), |V(F_1\cup F_2)|-2\}+\min\{k(F_1\cap F_2), |V(F_1\cap F_2)|-2\}.
\end{align}
On the other hand, by (\ref{eq:k}), 
\begin{align*}
k(F_1)+k(F_2)&\geq k(F_1\cup F_2)+k(F_1\cap F_2) \\
&\geq \min\{k(F_1\cup F_2), |V(F_1\cup F_2)|-2\}+\min\{k(F_1\cap F_2), |V(F_1\cap F_2)|-2\}. 
\end{align*}
Thus (\ref{eq:lem1-1-2}) implies that the equality holds everywhere, and in particular
\begin{align}
\label{eq:lem1-1-3}
k(F_1)+k(F_2)&= k(F_1\cup F_2)+k(F_1\cap F_2) \\ \label{eq:lem1-1-4}
k(F_1\cap F_2)&\leq |V(F_1\cap F_2)|-2.
\end{align}
%
By (\ref{eq:lem1-1-3}) and (\ref{eq:k}) we have $\chi(V(F_1\cap F_2))=\chi(V(F_1)\cap V(F_2))$; 
in particular, $c\in \chi( V(F_1\cap F_2))$.
Hence (\ref{eq:lem1-1-1}) and (\ref{eq:lem1-1-4}) imply that 
$F_1\cap F_2$ is $c$-tight.
This contradicts the minimality of $F_1$.
%
\end{proof}

\begin{lemma}
\label{lem:G'}
Let $(G,\chi)$ be a $g_1$-tight $k$-coloured graph and let $c$ be a colour.
Let $e$ be an edge in the inclusionwise minimal $(c,1)$-tight set if one exists and otherwise let $e$ be any edge.
Also let 
\begin{equation*}
G'=\begin{cases}
G-e & \text{if $k\leq |V|-2$,} \\
G & \text{otherwise,}
\end{cases}
\end{equation*}
and let $\chi'$ be the colouring obtained from $\chi$ by uncolouring vertices in $\chi^{-1}(c)$.
Then $(G',\chi')$ is $g_1$-tight.
\end{lemma}
\begin{proof}
Suppose that $k\geq |V|-1$.
Then $k(F)\geq |V(F)|-1$ for any nonempty $F\subseteq E$.
Hence after uncolouring vertices in $\chi^{-1}(c)$ we get $k(F)\geq |V(F)|-2$ for any nonempty $F\subseteq E$. This means that the colouring does not play any role in the count condition.
In other words $(G,\chi)$ is $g_1$-sparse if and only if $(G',\chi')$ is $g_1$-sparse.

We hence consider the case when  $k\leq |V|-2$.
Then it suffices to show that $G'$ is $g_1$-sparse.
Note that, when uncolouring $\chi^{-1}(c)$, $g_1(F)$ decreases by one for $F$ with $c\in \chi(V(F))$
and $g_1(F)$ remains unchanged for other $F$.
Therefore, if $G'$ is not $g_1$-sparse, $G$ should contain a $(c,1)$-tight set $F$ with $e\notin F$.
However the existence of such an $F$ contradicts Lemma~\ref{lem:1-2}.
\end{proof}

\begin{theorem}
\label{thm:1d}
Let $(G,\chi)$ be a $k$-coloured graph
and $(G,\chi, p)$ be a generic framework on $r\mathbb{S}^1$.
Then $(G,\chi, p)$ is isostatic if and only if $(G,\chi)$ is $g_1$-tight, i.e., 
\begin{itemize}
\item $|E|=|V|-1+\min\{k,|V|-2\}$, and 
\item $|F|\leq |V|-\omega(F)+\min\{k(F),|V(F)|-2\}$ for any nonempty $F\subseteq E$.
\end{itemize}
\end{theorem}
\begin{proof}
By Lemma~\ref{lem:necessity}, it suffices to show that,
if $G$ is $g_1$-tight,
there is an isostatic framework $(G,\chi,p)$ on $r\mathbb{S}^1$.

The proof is done by induction on $|V|+k$.

If there is a vertex of degree one or a colour-isolated vertex of degree two, then 
the statement follows by Lemma~\ref{lem:0extension}.
Hence we may assume that $G$ contains a non-colour-isolated vertex of degree two or 
a colour-isolated vertex of degree three by Lemma~\ref{lem:1-1}.

Suppose there is a colour-isolated vertex $v$ of degree three.
\begin{claim}
\label{claim:1}
1-reduction is possible at $v$, i.e., there are two neighbors $i, j\in N_G(v)$
such that $G-v+ij$ is $g_1$-sparse.
\end{claim}
\begin{proof}[Proof of Claim]
We use the proof technique given in \cite{jkt} based on the fact that the count condition induces a matroid.
Specifically, let $N_G(v)=\{i,j,k\}$ and assume that none of $G-v+ij, G-v+jk, G-v+ki$ is $g_1$-sparse.
Then $\{ij, jk, ki\}\subset {\rm cl}(E(G-v))$ where ${\rm cl}$ denotes the closure of ${\cal M}_1(V,\chi)$. 
Since $\{ij, jk, ki, vi, vj, vk\}$  is dependent in ${\cal M}_1(V,\chi)$, 
$vi\in {\rm cl}(\{ij, jk, ki,  vj, vk\})$. 
Hence $vi\in {\rm cl}(\{ij, jk, ki,  vj, vk\}\cup E(G-v))={\rm cl}(E(G-v)\cup \{vj, vk\})$.
Since $vi, vj, vk$ are in $G$, $E(G)$ is dependent, contradicting the $g_1$-sparsity of $G$.
\end{proof}
By induction $(G-v+ij,\chi|_{V-v})$ can be realized as an isostatic framework on $r\mathbb{S}^1$, and hence $(G,\chi)$ can also, by Lemma~\ref{lem:1extension1}.

Thus $G$ should contain a non-colour-isolated vertex $v$ of degree two.
Suppose all vertices of $N_G(v)\cup \{v\}$ are uncoloured.
Then one can again, using the argument in the claim, that
1-reduction is possible at $v$. (Note that the triangle on uncoloured three vertices is dependent.)
 This means that $G$ can be realized as an isostatic  framework 
by induction and Lemma~\ref{lem:1extension1}.

We hence assume that a vertex $u\in N_G(v)\cup \{v\}$ is coloured.
If $k\geq |V|-2$, then the claim follows from Lemma~\ref{lem:0extension} and induction.
Hence, we further assume $k<|V|-2$.
Let $(G,\chi,p)$ be a generic realization, and suppose that $(G,\chi,p)$ is not rigid.
Then by induction $(G-v,\chi, p)$ has one-degree of freedom,
and by Lemma~\ref{lem:0extension} $(G,\chi,p)$ is also a one-degree of freedom framework.
Hence by Lemma~\ref{lem:1extension2} 
there is a nontrivial infinitesimal motion $(\dot{p},\dot{r})$ of $(G,\chi,p)$ such that 
$\dot{r}(\chi(u))=0$.
This means that, even if the radius of the sphere associated with colour $\chi(u)$ is fixed, 
$(G,\chi,p)$ is still flexible.
Since this occurs for a generic choice of $r(\chi(u))$,
it also holds even for $r(\chi(u))=1$.
(Formally this follows from the symbolic rank of the matrix obtained from $R(G,\chi,p)$ by removing the column of $\chi(u)$.) 
In other words, considering a coloured graph $(G,\chi')$ with the $(k-1)$-colouring $\chi'$ obtained from $\chi$ by uncolouring all vertices in $\chi^{-1}(\chi(u))$, 
a generic framework $(G,\chi',p')$ on $r'\mathbb{S}^1$ is not rigid,
where $r'$ is the restriction of $r$ to ${\cal C}\setminus \{\chi(u)\}$.

We set $(G',\chi')$ as in the statement of Lemma~\ref{lem:G'}. 
Then $(G',\chi')$ is $g_1$-tight, and by induction $(G',\chi',p')$ is rigid on $r'\mathbb{S}^1$.
However, since $G'$ is a spanning subgraph of $G$, $(G,\chi',p')$ is rigid, a contradiction.
Thus we conclude that $(G,\chi,p)$ is rigid.
\end{proof}

One can easily check that, if each vertex has a distinct colour, then 
$|F|\leq |V|-\omega(F)+\min\{k(F),|V(F)|-2\}$ for any nonempty $F\subseteq E$ if and only if 
$|F|\leq 2|V(F)|-3$ for any nonempty $F\subseteq E$.
Therefore Theorem~\ref{thm:1d} is an interpolation theorem between the 1-dimensional characterisation of generic rigidity and Laman's theorem.

\subsection{Higher dimensional cases} \label{sec:higherdim}

If $d\geq 2$ then $g_d(F)\leq 0$ for $F$ with $|F|=1$.
So we shall consider the matroid induced by $f_d:=r_d+k$ instead of $g_d$.
Let ${\cal N}_d(V,\chi)$ be the matroid induced by $f_d$ on the edge set of the complete graph on $V$ with colouring $\chi$,
(where $E$ is independent iff $|F|\leq r_d(F)+k(F)$ for any $F\subseteq E$).
\begin{lemma}
\label{lem:spanning}
If $G=(V,E)$ is $f_d$-tight, then $G$ is rigid in $\mathbb{R}^d$.
\end{lemma}
\begin{proof}
Since $E$ is a base we have $|E|=dn-{d+1\choose 2}+k\leq r_d(E)+k$, implying 
$r_d(E)\geq dn-{d+1\choose 2}$. Thus $E$ attains the maximum rank in the rigidity matroid.
\end{proof}

The $k$-th elongation of the rigidity matroid ${\cal R}_d(V)$ is a matroid whose base family is the set of spanning sets of ${\cal R}_d(V)$ with cardinality equal to $dn-{d+1\choose 2}+k$.
It follows from Lemma~\ref{lem:spanning} that, if all vertices are uniformly coloured, 
${\cal N}_d(V,\chi)$ is the first elongation of ${\cal R}_d(V)$.
We conjecture that this matroid indeed characterizes the rigidity of expanding generic spherical frameworks in any dimension.
\begin{conjecture}
\label{con:one-variable}
Suppose $(G=(V,E),\chi)$ is uniformly coloured. 
Then $(G,\chi)$ is isostatic on $r\mathbb{S}^d$ if and only if $G$ is simple and $E$ is a base of the first elongation of ${\cal R}_d(V)$.
\end{conjecture}
For $d=2$ it is known that a simple graph whose edge set is a base of the first elongation of ${\cal R}_2(V)$ can be constructed by a sequence of 0-extensions and 1-extensions \cite{haas} (where such a graph is called a Laman-plus-one graph), and based on which Conjecture~\ref{con:one-variable} can be confirmed.
This can be easily extended to simple graphs whose edges sets are bases of ${\cal N}_2(V,\chi)$ for $k\leq 1$. On the other hand for $k\geq 2$ it does not seem to be possible to develop such a simple constructive characterisation even when $d=2$. (Constructing only {\em simple} graphs increases the complication.)
However it turns out that the proof technique of Theorem~\ref{thm:1d} can be directly applied even to the case when $d=2$ and $k\leq 2$.
To this end we shall first give a counterpart of Lemma~\ref{lem:1-2}.
\begin{lemma}
\label{lem:2-2}
Suppose $(G,\chi)$ is $f_d$-sparse.
Then for each colour $c$ either there is no $(c,d)$-tight set or 
$G$ contains  a unique inclusionwise minimal $(c,d)$-tight set.
\end{lemma}
\begin{proof}
Suppose there are two distinct inclusionwise minimal $c$-tight sets $F_1$ and $F_2$.
By the submodularity of $f_d$ 
we have 
\begin{align*}|F_1\cup F_2|+|F_1\cap F_2|&=|F_1|+|F_2|\geq
f_d(F_1)+f_d(F_2)\\ &\geq f_d(F_1\cup F_2)+f_d(F_1\cap F_2)
\\ &\geq |F_1\cup F_2|+|F_1\cap F_2|.\end{align*}
Hence the equality holds everywhere.
In particular, since $r_d$ and $k$ are nonnegative submodular, we have 
\begin{align}
\label{eq:lem2-2-1}
|F_1\cap F_2|&=f_d(F_1\cap F_2) \mbox{ and } \\ \label{eq:lem2-2-2}
k(F_1)+k(F_2)&=k(F_1\cup F_2)+k(F_1\cap F_2).
\end{align}
Comparing (\ref{eq:lem2-2-2}) with (\ref{eq:k}) we get $\chi(V(F_1\cap F_2))=\chi(V(F_1)\cap V(F_2))$;
in particular, $c\in \chi(V(F_1\cap F_2))$ and hence $F_1\cap F_2\neq \emptyset$.
Hence (\ref{eq:lem2-2-1}) implies that $F_1\cap F_2$ is $c$-tight, contradicting the minimality of $F_1$.
\end{proof}

\begin{theorem}
\label{thm:2d}
Suppose $(G=(V,E),\chi)$ is a $k$-coloured graph with $k\leq 2$ 
and $(G,\chi,p)$ is generic.
Then $(G,p)$ is isostatic on $r\mathbb{S}^2$ if and only if 
$G$ is a simple $f_2$-tight graph.
\end{theorem}
\begin{proof}
Since $|E|=2|V|-3+k\leq 2|V|-1$, the average degree is less than four,
and $G$ has a vertex of degree three.
If $G$ has a vertex of degree two or a colour-isolated vertex of degree three, 
then the statement follows by Lemma~\ref{lem:0extension}.
Hence assume that $G$ has a non-colour-isolated vertex $v$ of degree three. 

If all vertices in $N_G(v)\cup \{v\}$ are uncoloured, 
one can easily check by the argument in Claim \ref{claim:1} that 1-reduction is admissible so that the resulting graph is simple and $f_2$-tight. Hence the statement follows by induction and Lemma~\ref{lem:1extension1}.
(Note that an uncoloured $K_4$ is dependent.)

If some vertex in $N_G(v)\cup \{v\}$ is coloured, then 
 we can apply exactly the same argument as in the proof of Theorem~\ref{thm:1d} by utilising Lemma~\ref{lem:2-2} instead of Lemma~\ref{lem:1-2}. 
\end{proof}

A natural question is for how large $k$ the $f_2$-tightness characterizes the generic rigidity.
Unfortunately already for $k=3$ there is an example which is a simple $f_2$-tight graph but is not rigid, see Figure \ref{fig:bananas}. It is then easy to extend this example to any $k\geq 3$ colours using 0- and 1-extensions. In particular the famous double banana graph arises from the graph in Figure \ref{fig:bananas}(b) by a colour-isolated 0-extension followed by colouring a fifth vertex and adding a single edge.

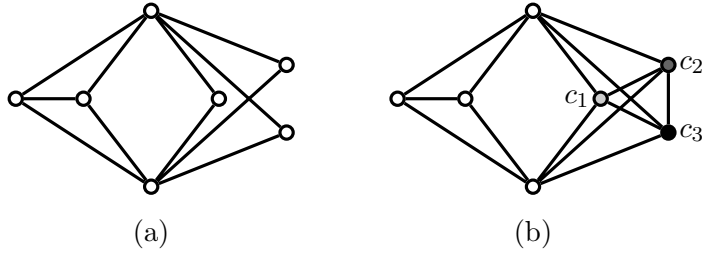
\begin{figure}[htp]
\begin{center}
\begin{tikzpicture}[very thick,scale=0.9]
\tikzstyle{every node}=[circle, draw=black, fill=white, inner sep=0pt, minimum width=5pt];

\node(p1) at (0,1.3) {};
\node(p2) at (-2,0) {};
\node(p3) at (0,-1.3) {};
\node(p4) at (-1,0) {};
\node(p5) at (1,0) {};
\node(p6) at (2,.5) {};
\node(p7) at (2,-.5) {};

\draw(p1)--(p2) -- (p3) -- (p4) -- (p2);

\draw(p4) -- (p1)--(p5)--(p3)--(p6)--(p1)--(p7)--(p3);

\node[rectangle,draw=white, fill=white](a) at (0,-2) {(a)};

\end{tikzpicture}
\hspace{1cm}
\begin{tikzpicture}[very thick,scale=0.9]
\tikzstyle{every node}=[circle, draw=black, fill=white, inner sep=0pt, minimum width=5pt];

\node(p1) at (0,1.3) {};
\node(p2) at (-2,0) {};
\node(p3) at (0,-1.3) {};
\node(p4) at (-1,0) {};
\node[fill=black!20!white](p5) at (1,0) {};
\node[fill=black!60!white](p6) at (2,.5) {};
\node[fill=black](p7) at (2,-.5) {};

\node[rectangle,draw=white, fill=white](l) at (0.65,0) {$c_1$};
\node[rectangle,draw=white, fill=white](l) at (2.35,0.5) {$c_2$};
\node[rectangle,draw=white, fill=white](l) at (2.35,-0.5) {$c_3$};

\draw(p1)--(p2) -- (p3) -- (p4) -- (p2);

\draw(p4) -- (p1)--(p5)--(p3)--(p6)--(p1)--(p7)--(p3);

\draw(p5)--(p6)--(p7)--(p5);

\node[rectangle,draw=white, fill=white](a) at (0,-2) {(b)};

\end{tikzpicture}
\end{center}
\vspace{-0.3cm}
\caption{Graphs of generic frameworks in $r\mathbb{S}^2$: (a) an isostatic uncoloured graph and (b) a 3-coloured graph (with colours $c_1,c_2,c_3$) which is $f_2$-tight  but is not isostatic.}
\label{fig:bananas}
\end{figure}

Note that the current proof of Conjecture~\ref{con:one-variable} for $d\leq 2$ (as a special case of Theorems \ref{thm:1d} and \ref{thm:2d}) relies on the existence of low degree vertices, and the proof cannot be extended to the case when $d\geq 3$ as in the  case for generic 3-dimensional rigidity. 
However it might be possible to show Conjecture~\ref{con:one-variable} by an algebraic argument.
Note that Conjecture~\ref{con:one-variable} holds once one can prove that, for any generic $(G,p)$ on $\mathbb{S}^d$ such that $E$ is a rigid circuit of the $d$-dimensional rigidity matroid,  a nonzero self-stress $\omega$ of $(G,p)$ satisfies $\sum_{v\in V} \omega(v)\neq 0$,
where, for a framework $(G,p)$ on $\mathbb{S}^d$, $\omega:V\cup E\rightarrow \mathbb{R}$ is said to be a self-stress if $\omega^{\top}R_{\mathbb{S}^{d}}(G,p)=0$ (by regarding $\omega$ as a $(|V|+|E|)$-dimensional vector).

\subsection{Rank formula}

We conclude this section by deriving an explicit rank formula for ${\cal N}_2(V,\chi)$.
Since $f_d$ is a nonnegative submodular function, the rank of an edge set $E$ in the induced matroid ${\cal N}_d(V,\chi)$ can be written as 
\begin{equation}
\label{eq:r1}
\min\{|E\setminus F|+f_d(F) \mid F\subseteq E\}.
\end{equation}
For $d=2$, a simple formula of the rank function $r_2$ of the generic rigidity matroid ${\cal R}_2(V)$ is known~\cite{ly}, and combining it with (\ref{eq:r1}) we get an explicit rank formula of ${\cal N}_2(V,\chi)$.
By Theorem~\ref{thm:2d}, this formula gives a combinatorial description of the rank of the rigidity matrix $R(G,\chi,p)$ for generic $p$ and $k\leq 2$.

To see this, let us give the rank formula in \cite{ly}.
For a graph $G=(V,E)$, a {\em cover} is a family ${\cal X}=\{X_1,\dots, X_k\}$ of subsets of $V$ 
such that $E=\bigcup_{i=1}^k E(X_i)$, where $E(X_i)$ denotes the set of edges induced by $X_i$. 
The cover ${\cal X}$ is \emph{$1$-thin} if every pair of sets in ${\cal X}$ intersect in at most one vertex.
The following is the formula by Lov{\'a}sz and Yemini~\cite{ly}:
\begin{equation}
\label{eq:r2}
r_2(E)=\min\left\{\sum_{X_i\in {\cal X}} (2|X_i|-3) ~\Bigg| \text{ a 1-thin cover {\cal X} of $G=(V,E)$}\right\}.
\end{equation}
See, e.g., \cite{jordan} for applications of the rank formula.
To extend this, we introduce one additional condition on  covers. For $X\subseteq V$, a cover ${\cal X}$ is said to be {\em $X$-restricted} if 
$|X_i|\geq 3$ implies $X_i\subseteq X$ for each $X_i\in {\cal X}$.
Then we have the following.
\begin{theorem}
Let $(G = (V, E), \chi)$ be a $k$-coloured graph with $k\leq 2$ 
and $(G,\chi, p)$ be a generic framework on $r\mathbb{S}^2$.
Then 
\[
{\rm rank}\ R(G,\chi,p)=
\min\left\{ \sum_{X_i\in {\cal X}} (2|X_i|-3)+k(X) \Bigg|
\begin{array}{cc}
X\subseteq V \\
\text{ an $X$-restricted 1-thin cover } {\cal X} \text{ of $G$} 
\end{array}
\right\}
\]
where $k(X)$ denotes the number of colors in $X$.
\end{theorem}
\begin{proof}
By Theorem~\ref{thm:2d}, 
the rank of $R(G,\chi,p)$ is equal to the rank of $E$ in ${\cal N}_d(V,\chi)$.
By (\ref{eq:r1}),
\begin{equation}
\label{eq:r3}
{\rm rank}\ R(G,\chi,p)=\min\{|E\setminus F|+r_2(F)+k(F)\mid F\subseteq E\}.
\end{equation}
Clearly the minimizer $F$ of (\ref{eq:r3}) can be taken so that $F$ is induced, i.e., $F$ is the edge set of the subgraph induced by some $X\subseteq V$.
Let ${\cal X}'$ be a  1-thin cover minimizing (\ref{eq:r2}) for $G[X]$,
and let ${\cal X}={\cal X}'\cup \{\{u,v\}\mid e=uv\in E\setminus F\}$.
Then ${\cal X}$ is an $X$-restricted  1-thin cover of $G$.
Also, since $F$ is induced, $|E\setminus F|=\sum_{X_i\in {\cal X}\setminus {\cal X}'}(2|X_i|-3)$.
Thus, by (\ref{eq:r2}) and (\ref{eq:r3}), we get the stated formula.
\end{proof}

%
%
%
%

\section{Symmetry-induced motions}
\label{sec:sym}

The frameworks appearing in our suggested applications typically exhibit symmetry and it turns out that symmetry frequently induces (continuous) flexibility in frameworks which are rigid in $r\mathbb{S}^d$ if vertices are placed generically (see for example Figure~\ref{fig:2D}). Moreover, these added motions typically preserve the original symmetry of the framework throughout the path. Thus, in the following, we establish symmetry-adapted  combinatorial  counts which allow us to detect such symmetry-induced motions. 

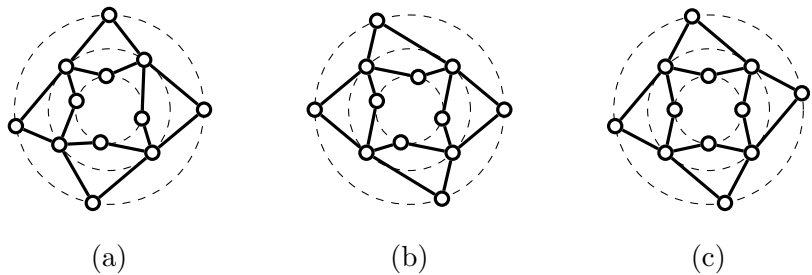
\begin{figure}[htp]
\begin{center}
\begin{tikzpicture}[very thick,scale=0.9]
\tikzstyle{every node}=[circle, draw=black, fill=white, inner sep=0pt, minimum width=5pt];
\filldraw[fill=white, draw=black, thin, dashed](0,0)circle(1.4cm);
\filldraw[fill=white, draw=black, thin, dashed](0,0)circle(0.9cm);
\filldraw[fill=white, draw=black, thin, dashed](0,0)circle(0.5cm);

\node(p1) at (95:0.5cm) {};
\node(p2) at (165:0.5cm) {};
\node(p3) at (255:0.5cm) {};
\node(p4) at (345:0.5cm) {};

\node(pp1) at (55:0.9cm) {};
\node(pp2) at (135:0.9cm) {};
\node(pp3) at (215:0.9cm) {};
\node(pp4) at (315:0.9cm) {};

\draw(p1)--(pp2);
\draw(p2)--(pp2);
\draw(p2)--(pp3);
\draw(p3)--(pp3);
\draw(p3)--(pp4);
\draw(p4)--(pp4);
\draw(p4)--(pp1);
\draw(p1)--(pp1);

\node(ppp1) at (90:1.4cm) {};
\node(ppp2) at (190:1.4cm) {};
\node(ppp3) at (260:1.4cm) {};
\node(ppp4) at (0:1.4cm) {};

\draw(pp3)--(ppp2);
\draw(pp2)--(ppp2);
\draw(pp4)--(ppp3);
\draw(pp3)--(ppp3);
\draw(pp1)--(ppp4);
\draw(pp4)--(ppp4);
\draw(pp2)--(ppp1);
\draw(pp1)--(ppp1);
\node[rectangle,draw=white, fill=white](a) at (0,-2.2) {(a)};

\end{tikzpicture}
\hspace{1cm}
\begin{tikzpicture}[very thick,scale=0.9]
\tikzstyle{every node}=[circle, draw=black, fill=white, inner sep=0pt, minimum width=5pt];
\filldraw[fill=white, draw=black, thin, dashed](0,0)circle(1.4cm);
\filldraw[fill=white, draw=black, thin, dashed](0,0)circle(0.9cm);
\filldraw[fill=white, draw=black, thin, dashed](0,0)circle(0.5cm);

\node(p1) at (75:0.5cm) {};
\node(p2) at (165:0.5cm) {};
\node(p3) at (255:0.5cm) {};
\node(p4) at (345:0.5cm) {};

\node(pp1) at (45:0.9cm) {};
\node(pp2) at (135:0.9cm) {};
\node(pp3) at (225:0.9cm) {};
\node(pp4) at (315:0.9cm) {};

\draw(p1)--(pp2);
\draw(p2)--(pp2);
\draw(p2)--(pp3);
\draw(p3)--(pp3);
\draw(p3)--(pp4);
\draw(p4)--(pp4);
\draw(p4)--(pp1);
\draw(p1)--(pp1);

\node(ppp1) at (110:1.4cm) {};
\node(ppp2) at (180:1.4cm) {};
\node(ppp3) at (290:1.4cm) {};
\node(ppp4) at (0:1.4cm) {};

\draw(pp3)--(ppp2);
\draw(pp2)--(ppp2);
\draw(pp4)--(ppp3);
\draw(pp3)--(ppp3);
\draw(pp1)--(ppp4);
\draw(pp4)--(ppp4);
\draw(pp2)--(ppp1);
\draw(pp1)--(ppp1);
\node[rectangle,draw=white, fill=white](a) at (0,-2.2) {(b)};

\end{tikzpicture}
\hspace{1cm}
\begin{tikzpicture}[very thick,scale=0.9]
\tikzstyle{every node}=[circle, draw=black, fill=white, inner sep=0pt, minimum width=5pt];
\filldraw[fill=white, draw=black, thin, dashed](0,0)circle(1.4cm);
\filldraw[fill=white, draw=black, thin, dashed](0,0)circle(0.9cm);
\filldraw[fill=white, draw=black, thin, dashed](0,0)circle(0.5cm);

\node(p1) at (90:0.5cm) {};
\node(p2) at (180:0.5cm) {};
\node(p3) at (270:0.5cm) {};
\node(p4) at (0:0.5cm) {};

\node(pp1) at (45:0.9cm) {};
\node(pp2) at (135:0.9cm) {};
\node(pp3) at (225:0.9cm) {};
\node(pp4) at (315:0.9cm) {};

\draw(p1)--(pp2);
\draw(p2)--(pp2);
\draw(p2)--(pp3);
\draw(p3)--(pp3);
\draw(p3)--(pp4);
\draw(p4)--(pp4);
\draw(p4)--(pp1);
\draw(p1)--(pp1);

\node(ppp1) at (100:1.4cm) {};
\node(ppp2) at (190:1.4cm) {};
\node(ppp3) at (280:1.4cm) {};
\node(ppp4) at (10:1.4cm) {};

\draw(pp3)--(ppp2);
\draw(pp2)--(ppp2);
\draw(pp4)--(ppp3);
\draw(pp3)--(ppp3);
\draw(pp1)--(ppp4);
\draw(pp4)--(ppp4);
\draw(pp2)--(ppp1);
\draw(pp1)--(ppp1);
\node[rectangle,draw=white, fill=white](a) at (0,-2.2) {(c)};

\end{tikzpicture}
\end{center}
\vspace{-0.3cm}
\caption{Frameworks in $r\mathbb{S}^1$ with $k=3$ colours: (a) a non-symmetric generic realization which is infinitesimally rigid (over-braced by $2$); (b) a half-turn-symmetric realization which is still infinitesimally rigid and $\mathbb{Z}_2$-symmetric isostatic; (c) a realization with 4-fold rotational symmetry which is (continuously) flexible with a symmetry-preserving flex.}
\label{fig:2D}
\end{figure}

\subsection{Symmetric coloured graphs}

Let $(G,\chi)$ be a coloured graph. An {\em automorphism} of $(G,\chi)$ is a permutation $\pi:V\rightarrow V$ such that $\{i,j\}\in E$ if and only if $\{\pi(i),\pi(j)\}\in E$, and $\chi(\pi(i))=\chi(i)$ for all $i\in V$.  The   group  of all automorphisms of  $(G,\chi)$ is denoted by $\textrm{Aut}(G,\chi)$. For an abstract group $\Gamma$, we say that $(G,\chi)$ is {\em $\Gamma$-symmetric} if there exists a group action  $\theta:\Gamma \to \textrm{Aut}(G,\chi)$.
 In the following, we will always assume that the action $\theta$ is free on the vertex set of $G$, and we will  omit $\theta$ if it is clear from the context. We will then simply write $\gamma i$ instead of $\theta(\gamma)(i)$.  

 The {\em quotient graph} of a $\Gamma$-symmetric coloured graph $(G,\chi)$ is the coloured multigraph $(G/\Gamma,\chi_0)$ whose vertex set is the set $V/\Gamma$
of vertex orbits and whose edge set is the set $E/\Gamma$ of edge orbits. The function $\chi_0$ assigns to the vertex orbit $\Gamma i=\{\gamma i \mid \gamma \in \Gamma\}$ the colour $\chi(i)$. Note that an edge orbit may be represented by a loop in $G/\Gamma$.

The {\em quotient $\Gamma$-gain graph} of a $\Gamma$-symmetric coloured graph $(G,\chi)$ is the pair\break $((G_0,\chi_0),\psi)$, where  $(G_0,\chi_0)=((V_0,E_0),\chi_0)$ is the quotient graph of $(G,\chi)$ with an orientation on the edges, and $\psi:E_0\to \Gamma$ is defined as follows. 
Each edge orbit $\Gamma e$ connecting $\Gamma i$ and $\Gamma j$ in $G/\Gamma$ can be written as $\{\{\gamma i,\gamma \circ\alpha j\}\mid \gamma\in \Gamma \}$ for a unique $\alpha\in \Gamma$. For each $\Gamma e$, orient $\Gamma e$ from $\Gamma i$ to $\Gamma j$ in $G/\Gamma$ and assign to it the gain $\alpha$. Then $E_0$ is the resulting set of oriented edges, and $\psi$ is the corresponding gain assignment.
(See \cite{jkt} for details.)

Suppose $\Gamma$ is an abstract multiplicative group. 
A closed walk $C=v_1,e_1,v_2,\dots,v_k,e_k,\break v_1$  in a quotient $\Gamma$-gain graph $((G_0,\chi_0),\psi)$ is called {\em balanced} if $\psi(C)=\Pi_{i=1}^k \psi(e_i)^{{\rm sign}(e_i)}=\textrm{id}$,
where ${\rm sign}(e_i)=1$ if $e_i$ is directed from $v_i$ to $v_{i+1}$, and ${\rm sign}(e_i)=-1$  otherwise. We say that an edge subset $F_0\subseteq E_0$ is {\em balanced} if all closed walks in $F_0$ are balanced; otherwise it is called {\em unbalanced}.

For a subset $F_0\subseteq E_0$, and a vertex $v$ of the vertex set $V(F_0)\subseteq V_0$ induced by $F_0$, the \emph{subgroup induced by $F_0$} is the subgroup $\langle F_0\rangle_{\psi,v}=\{\psi(C)|\, C\in \pi_1(F_0,v)\}$ of $\Gamma$, where $\pi_1(F_0,v)$ is the set of closed walks starting at $v$ using only edges of $F_0$, and $\pi_1(F_0,v)=\emptyset$  if $v\notin V(F_0)$. 
If $\psi$ is clear from the context, then we also simply write $\langle F_0\rangle_{v}$ for $\langle F_0\rangle_{\psi,v}$. 

\subsection{Symmetric frameworks in $r\mathbb{S}^d$}

Let $(G,\chi)$ be a   $\Gamma$-symmetric coloured graph (with respect to the group action $\theta:\Gamma \to \textrm{Aut}(G)$) and let $\tau:\Gamma\rightarrow O(\mathbb{R}^{d+1})$ be a group representation. A framework $(G,\chi,p)$ is called {\em $\Gamma$-symmetric} (with respect to $\theta$ and  $\tau$) if \begin{equation*}
\label{eq:symmetricfw}
\tau(\gamma) p_i=p_{\gamma i}\qquad \text{for all } \gamma\in \Gamma \text{ and all } i\in V.
\end{equation*}

An infinitesimal motion $(\dot{p},\dot{r})$ of a $\Gamma$-symmetric framework $(G,\chi,p)$ in $r\mathbb{S}^d$ (with respect to $\theta$ and $\tau$) is called {\em $\Gamma$-symmetric} if it satisfies $\tau(\gamma)\dot{p}_i=\dot{p}_{\gamma i}$ for all  $\gamma\in \Gamma$ and $ i\in V$, i.e., if the velocity vectors of the infinitesimal motion exhibit the same symmetry as the framework.
$(G,\chi,p)$ is {\em $\Gamma$-symmetric infinitesimally rigid} if  every $\Gamma$-symmetric infinitesimal motion is trivial.

The following symmetric analog of the rigidity matrix (recall (\ref{eq:rigmat})) can be used to analyze the $\Gamma$-symmetric infinitesimal rigidity properties of a framework $(G,\chi,p)$ in $r\mathbb{S}^d$.
To define this matrix we fix a representative vertex $i$ of each vertex orbit $\Gamma i$,
and define the {\em quotient} of $p$ to be $\tilde{p}:V_0\rightarrow \mathbb{R}^{d+1}$,
so that there is a one-to-one correspondence between $p$ and $\tilde{p}$ given by
$p(i)=\tilde{p}(C(i))$, where $C:(G,\chi)\rightarrow (G_0,\chi_0)$ is the {\em covering map} defined by $C(\gamma i)=i$ and $C(\{\gamma i,\gamma\psi(e)j\})=(i,j)$.

Let $((G_0,\chi_0),\psi)$ be the  quotient $\Gamma$-gain graph of $(G,\chi)$.  The \emph{orbit rigidity matrix} $O((G_0,\chi_0),\psi,\tilde{p})$  of the framework $(G,\chi,p)$  is the matrix of size $(|E_0|+|V_0|) \times ((d+1)|V_0|+k)$ consisting of four blocks
\begin{equation}\label{eq:orbitrigmat}
O((G_0,\chi_0),\psi,\tilde{p})=\begin{pmatrix}
O(G_0,\psi,\tilde{p}) & 0 \\
S(G_0,\tilde{p}) & V(G_0,\chi_0,\tilde{p})
\end{pmatrix},
\end{equation}
where $O(G_0,\psi,\tilde{p})$ is the orbit rigidity matrix of the framework $(G,p)$ in $\mathbb{R}^{d+1}$ introduced first in \cite{BSWW}.
The explicit entries of  $O(G_0,\psi,\tilde{p})$ have the following form: the row corresponding to the edge $e=(i,j)\in E_0$ is given by
\[\begin{array}{ccccc}
       & \overbrace{\hspace{25mm}}^{i} & & \overbrace{\hspace{30mm}}^{j} & \\
       \big(0\dots0 & \tilde{p}(i)-\tau(\psi(e))\tilde{p}(j) & 0\dots0 & \tilde{p}(j)-\tau(\psi(e))^{-1}\tilde{p}(i) & 0\dots0\big )  \\
    \end{array}\]
    if $e$ is not a loop, and if $e$ is a loop at $i$ then the row corresponding to $e$ is given by
\[\begin{array}{ccccc}
       & \overbrace{\hspace{55mm}}^{i} & &  & \\
       \big (0\dots0 & 2\tilde{p}(i)-\tau(\psi(e))\tilde{p}(i)-\tau(\psi(e))^{-1}\tilde{p}(i)& 0\dots0 & 0 & 0\dots0\big)  \\
    \end{array}.\]
 Note that if all vertices are uncoloured, then (\ref{eq:orbitrigmat}) is the    standard spherical orbit rigidity matrix $O_{\mathbb{S}^d}(G_0,\psi,\tilde{p})$ for frameworks on $\mathbb{S}^{d}$ with a fixed radius discussed in \cite{BSWWSphere}.

We say that a $\Gamma$-symmetric framework $(G,\chi,p)$ on $r\mathbb{S}^d$ is called {\em $\Gamma$-symmetric independent} if $O_{\mathbb{S}^d}(G_0,\psi,\tilde{p})$ is row independent while it is called {\em $\Gamma$-symmetric isostatic} if it is $\Gamma$-symmetric infinitesimally rigid and independent.	
Also $(G,\chi,p)$ is called {\em $\Gamma$-generic} if the coordinates of $\tilde{p}$ and $r$ do not satisfy any nonzero polynomial except for those belonging to the ideal generated by the defining polynomials of concentric spheres $r\mathbb{S}^d$.
 Combinatorial characterisations of $\Gamma$-generic infinitesimally rigid frameworks on the fixed sphere $\mathbb{S}^{2}$ have been established for various groups in \cite{ns}, by extending the results for $\mathbb{R}^2$ \cite{jkt,mt,mt1}.

Finally we note that by extending the results in \cite{BS1} it is easy to see that for a  $\Gamma$-generic framework in $r\mathbb{S}^d$, a $\Gamma$-symmetric infinitesimal motion always extends to a continuous  motion which preserves the symmetry of the framework throughout the path. 

\subsection{Symmetry-adapted combinatorial counts}

For a subgroup $\mathcal{S}$ of the $d+1$-dimensional orthogonal group $O(\mathbb{R}^{d+1})$, a (column) vector $t\in \mathbb{R}^{d+1}$ is said to be {\em $\mathcal{S}$-invariant} if $St=t$ for every $S\in \mathcal{S}$. The dimension of the space of $\mathcal{S}$-invariant vectors is denoted by $t_{\mathcal{S}}$. 
 (Note that we may think of $t_{\mathcal{S}}$ as the dimension of the space of infinitesimal translations corresponding to the trivial irreducible representation of $\mathcal{S}$.) 


A straight-forward adaptation of the proof of Proposition~\ref{prop:trivial} yields the following result.

\begin{prop}\label{prop:symcount}
Let $\Gamma$ be a group, and let $(G,\chi,p)$ be a $\Gamma$-symmetric framework (with respect to $\theta:\Gamma\to \textrm{Aut}(G)$ and $\tau:\Gamma\to O(\mathbb{R}^{d+1})$) on $r\mathbb{S}^d$, where $(G,\chi)$ is a $k$-coloured graph, $\Gamma$ acts freely on the vertices of $G$, and $p(V)$ affinely spans $\mathbb{R}^{d+1}$.  Then we have
$${\rm rank }\ O((G_0,\chi_0),\psi,\tilde{p})\leq {\rm rank }\ O_{\mathbb{S}^d}(G_0,\psi,\tilde{\overline{p}})+\min\{k,|V_0|-t_{\tau(\Gamma)}\},
$$
 where $(G,\overline{p})$ is the framework on the (fixed) sphere $\mathbb{S}^d$ obtained from $(G,\chi,p)$ by central projection, and  $O_{\mathbb{S}^d}(G_0,\psi,\tilde{\overline{p}})$ is the standard spherical orbit rigidity matrix of  $(G,\overline{p})$.
\end{prop}

For a positive integer $d$ and a subgroup  $\mathcal{S}$ of the $d+1$-dimensional orthogonal group $O(\mathbb{R}^{d+1})$,
let $r_{d,\mathcal{S}}$ denote the rank function of the $d$-dimensional $\Gamma$-generic spherical rigidity matroid, which is defined in terms of the rank of the $d$-dimensional spherical $\Gamma$-symmetric orbit rigidity matrix, where the  radius of the sphere is assumed to be fixed.
By Proposition~\ref{prop:symcount}, we then  have the following count condition for a $\Gamma$-generic framework in $r\mathbb{S}^d$ to be $\Gamma$-symmetric independent.

\begin{prop}
Let $(G,\chi, p)$ be a $\Gamma$-generic framework (with respect to $\theta:\Gamma\to \textrm{Aut}(G)$ and $\tau:\Gamma\to O(\mathbb{R}^{d+1})$) on $r\mathbb{S}^d$, where $(G,\chi)$ is a $k$-coloured graph. Further, let
$((G_0,\chi_0), \psi)$ be the quotient $\Gamma$-gain graph of $(G,\chi)$.
If $(G,\chi, p)$ is $\Gamma$-symmetric independent, then for every $F_0\subseteq E_0$ with $|V(F_0)|\geq d$, and $v\in V(F_0)$, 
we have 
$$|F_0|\leq r_{d, \tau(\Gamma)}(F_0)+\min\{k(F_0), |V(F_0)|-t_{\tau(\langle F_0\rangle_v)}\}.$$
 \end{prop}


This leads to the following necessary conditions for $\Gamma$-symmetric frameworks to be $\Gamma$-symmetric isostatic 
in $r\mathbb{S}^1$. Analogously as before,  $\omega(F_0)$ denotes the number of connected components in $(V_0,F_0)$.

\begin{corollary}\label{cor:symcounts} Let $(G,\chi, p)$ be a $\Gamma$-generic framework (with respect to $\theta:\Gamma\to \textrm{Aut}(G)$ and $\tau:\Gamma\to O(\mathbb{R}^{2})$), where $(G,\chi)$ is a $k$-coloured graph.  If $(G,\chi, p)$ is $\Gamma$-symmetric isostatic, then the following hold.
\begin{itemize}
\item[(a)] If  $\Gamma=\mathbb{Z}_n$ and $\tau(\Gamma)$ describes $n$-fold rotational symmetry in the plane, then the quotient $\Gamma$-gain graph $((G_0,\chi_0), \psi)$ satisfies
\begin{itemize}
\item[(i)] $|E_0|=|V_0|-1+k$,
\item[(ii)] $|F_0|\leq \begin{cases} |V_0|-\omega(F_0)+\min\{k(F_0), |V(F_0)|-2\}, & \text{(if $F_0\subseteq E_0$ is balanced)}, \\
|V_0|-\omega(F_0)+k(F_0), & \text{(if $F_0\subseteq E_0$ is unbalanced)},
\end{cases}
$
\end{itemize}
\item[(b)] If $\Gamma=\mathbb{Z}_2$ and $\tau(\Gamma)$ describes reflectional symmetry in the plane, then the quotient $\Gamma$-gain graph $((G_0,\chi_0), \psi)$ satisfies
\begin{itemize}
\item[(i)] $|E_0|=|V_0|+\min\{k, |V_0|-1\}$,
\item[(ii)] $|F_0|\leq \begin{cases} |V_0|-\omega(F_0)+\min\{k(F_0), |V(F_0)|-2\}, & \text{(if $F_0\subseteq E_0$ is balanced)}, \\
|V(F_0)|+\min\{k(F_0), |V(F_0)|-1\} & \text{(if $F_0\subseteq E_0$ is unbalanced)}.
\end{cases}
$
\end{itemize}
\item[(c)] If $\Gamma$ is a dihedral group and $\tau(\Gamma)$ describes dihedral symmetry in the plane, then the quotient $\Gamma$-gain graph $((G_0,\chi_0), \psi)$ satisfies
\begin{itemize}
\item[(i)] $|E_0|=|V_0|+k$,
\item[(ii)] $|F_0|\leq \begin{cases} |V_0|-\omega(F_0)+\min\{k(F_0), |V(F_0)|-2\}, & \text{(if $F_0\subseteq E_0$ is balanced)}, \\
|V_0|-\omega(F_0)+k(F_0), & \text{(if $\tau(\langle F_0 \rangle_v)$ is a rotation group)},\\
|V(F_0)|+\min\{k(F_0), |V(F_0)|-1\}, & \text{(if $\tau(\langle F_0 \rangle_v)$ is a reflection group)},\\
|V(F_0)|+k(F_0), & \text{(if $\tau(\langle F_0 \rangle_v)$ is dihedral)}.
\end{cases}
$
\end{itemize}

\end{itemize}



\end{corollary}

Consider for example the frameworks in Figure~\ref{fig:2D}. The underlying $3$-coloured graph $(G,\chi)$ satisfies $|V|=12$ and
 $|E|=16$. Thus, $|E|=(|V|-1+3) + 2$, and hence  $(G,\chi)$ is  overbraced by $2$. In fact, by Theorem~\ref{thm:1d}, generic realizations of $(G,\chi)$ are rigid. 
If $(G,\chi)$ is realized with half-turn symmetry (see Figure~\ref{fig:2D}(b)) we obtain the counts
$|V_0|=6$, and $|E_0|=8$ for the quotient $\mathbb{Z}_2$-gain graph. Thus,  $|E_0|=|V_0|+ 2$, i.e., the framework counts to be $\mathbb{Z}_2$-symmetric isostatic.
Finally, if $(G,\chi)$ is realized generically with  4-fold rotational symmetry (see Figure~\ref{fig:2D}(c)), then we obtain the counts  $|V_0|=3$ and $|E_0|=4$, and hence $|E_0|=|V_0|+ 1$. Therefore, by Corollary~\ref{cor:symcounts},  the framework is guaranteed to have a non-trivial $\mathbb{Z}_4$-symmetric infinitesimal motion, which also extends to a symmetry-preserving continuous motion.

For all the rotational groups, the reflection group, and the dihedral groups of order $2n$, where $n$ is odd, we conjecture that the counts in Corollary~\ref{cor:symcounts} are also sufficient for $\Gamma$-generic frameworks in $r\mathbb{S}^1$ to be $\Gamma$-symmetric isostatic. For the dihedral groups of order  $2n$, where $n$ is even, however, there exist simple counterexamples. Bottema's mechanism \cite{bottema}, for example, has dihedral symmetry of order 4 and the corresponding quotient gain graph satisfies the counts in Corollary~\ref{cor:symcounts} for $k=2$ colours (see also \cite{jkt,BSWW}).

Finally, we also provide a sample result for frameworks in $r\mathbb{S}^2$. The corresponding results for other symmetry groups are obtained analogously.

\begin{corollary}\label{cor:2d}  Let $(G,\chi, p)$ be a $\mathbb{Z}_q$-generic framework (with respect to $\theta:\mathbb{Z}_q\to \textrm{Aut}(G)$ and $\tau:\mathbb{Z}_q\to O(\mathbb{R}^{3})$), where $\tau(\mathbb{Z}_q)$ describes rotational symmetry in 3-space, and   $(G,\chi)$ is a $k$-coloured graph. If $(G,\chi, p)$ is $\mathbb{Z}_q$-symmetric isostatic, then the  quotient $\mathbb{Z}_q$-gain graph $((G_0,\chi_0), \psi)$ satisfies
\begin{itemize}
\item[(i)] $|E_0|=2|V_0|-1+\min\{k, |V_0|-1\}$,
\item[(ii)]  
$|F_0|\leq 
\begin{cases} r_{2,\tau(\mathbb{Z}_q)}(F_0)+\min\{k(F_0), |V(F_0)|-3\}, & \text{(if $F_0\subseteq E_0$ is balanced)}, \\
r_{2,\tau(\mathbb{Z}_q)}(F_0)+\min\{k(F_0), |V(F_0)|-1\}, & \text{(if $F_0\subseteq E_0$ is unbalanced)}.
\end{cases}
$
\end{itemize}
\end{corollary}

By \cite{ns},  $|F_0|\leq r_{2,\tau(\mathbb{Z}_q)}(F_0)$ for every $F_0\subseteq E_0$ if and only if 
$$|F_0|\leq \begin{cases}
2|V(F_0)|-3 & \text{(if $F_0$ is balanced)}, \\
2|V(F_0)|-1 & \text{(if $F_0$ is unbalanced)}.
\end{cases}
$$
Hence it is natural to ask whether the subgraph counts in Corollary~\ref{cor:2d}(ii) can also be written as
$$|F_0|\leq \begin{cases} 2|V(F_0)|-3+\min\{k(F_0), |V(F_0)|-3\}, & \text{(if $F_0$ is balanced)}, \\
2|V(F_0)|-1+\min\{k(F_0), |V(F_0)|-1\}, & \text{(if $F_0$ is unbalanced).}
\end{cases}
$$
This also gives a necessary count but it is weaker. Consider for example a $1$-coloured $\mathbb{Z}_q$-quotient gain graph 
 consisting of two disjoint triangles  together with one loop at each vertex: $V_0=\{1,2,3,4,5,6\}$ and $E_0=\{11, 22, 33, 44, 55, 66, 12,23,31, 45, 56, 64\}$, where all loops have a non-trivial gain and all non-loops have trivial gain. This gain graph satisfies the weaker counts, but violates the counts in Corollary~\ref{cor:2d}.

For the same reason, the subgraph counts for balanced subgraphs or unbalanced subgraphs that induce a rotational symmetry group in Corollary~\ref{cor:symcounts} are expressed in terms of $|V_0|$ and $\omega(F_0)$ instead of $|V(F_0)|$.

We conjecture that the counts in Corollary~\ref{cor:2d} are also sufficient for  $\mathbb{Z}_q$-generic frameworks in $r\mathbb{S}^2$ to be $\mathbb{Z}_q$-symmetric isostatic for the uniformly coloured case, where $\mathbb{Z}_q$ describes rotational symmetry in 3-space.

\section{Concluding remarks} \label{sec:concl}

1) Our main results are robust under mild generalisations. For example, with our proof technique, it does not matter if our initial coloured graph $(G,\chi)$ is realised with all vertices of colour $c$ on $r(c)\mathbb{S}^d$ or on concentric $d$-spheres nor does it matter if all vertices start on the unit $d$-sphere. The key is that vertices with the same colour change radius at the same rate.
While these generalisations are equivalent generically, there will be subtle points about how they connect geometrically.  For example, if a  framework on concentric circles in the plane contains a triangle, then special choices of the radii for these three vertices will make the triangle collinear, hence losing its infinitesimal rigidity.
\medskip



\medskip \noindent
2) It is possible to replace the $d$-sphere with any algebraic set of dimension $d$ and repeat our analysis. As an example, an analogue of Theorem \ref{thm:2d} for concentric cylinders might be useful for understanding certain bacteriophages \cite{Pir}. The key change is to the submatrix $S(G,p)$ and hence to the dimension of the space of trivial motions.
The uncoloured case, for surfaces, has already been considered \cite{nop,nop2}. However it is easy to extend Figure \ref{fig:bananas}(b) to give counterexamples for any set of concentric surfaces with 3 independently variable radii. 

\medskip \noindent
3) A spherical framework on $r\mathbb{S}^d$ may be modeled as a `coned' framework in Euclidean $(d+1)$-space by adding a new cone vertex $v_0$ at the center of the sphere and by adding a new cone edge from $v_0$ to each of the original vertices. The sets of cone edges corresponding to the spheres in $r\mathbb{S}^d$ may change length at the same rate independently of each other. More generally, one could study the generic rigidity of arbitrary (i.e., not necessarily coned) Euclidean frameworks for which arbitrary subsets of edges are allowed to change length in a coordinated fashion independently of each other. This question is currently under investigation and some initial results have recently been presented at the workshop on `Geometric Rigidity Theory and Applications' at the ICMS in Edinburgh.

\end{document}